\documentclass[a4paper]{article}
\usepackage[T1]{fontenc}
\usepackage{mathptmx}
\usepackage{amsmath,amssymb,amsthm,amsbsy}
\usepackage{mathrsfs}
\usepackage{graphics,graphicx,color}
\usepackage{hyperref}
\usepackage{enumerate}
\usepackage{xcolor}
\usepackage{darkmode}

\usepackage{orcidlink}
\usepackage{setspace}
\usepackage[pagewise,displaymath, mathlines]{lineno} 

\newtheorem{lemma}{Lemma}[section]
\newtheorem{theorem}[lemma]{Theorem}
\newtheorem{corollary}[lemma]{Corollary}
\newtheorem{proposition}[lemma]{Proposition}
\newtheorem{example}[lemma]{Example}
\newtheorem{remark}[lemma]{Remark}

\newcommand{\im}{\mathrm{im}}

\def\ker{\mathop{\rm ker}}

\def\im{\mathop{\rm im}}

\def\fix{\mathop{\rm Fix}}
\newcommand{\rank}{\mathrm{rank}}

\numberwithin{equation}{section}

\begin{document}
\begin{center}
\textbf{\large{Semigroups of linear transformations whose restrictions belong to a general linear group}}\\
\vspace{0.5 cm} Kritsada Sangkhanan \orcidlink{0000-0002-1909-7514}
\end{center}


\begin{abstract}
	Let $V$ be a vector space and $U$ a fixed subspace of $V$. We denote the semigroup of all linear transformations on $V$ under composition of functions by $L(V)$. In this paper, we study the semigroup of all linear transformations on $V$ whose restrictions belong to the general linear group $GL(U)$, denoted by $L_{GL(U)}(V)$. More precisely, we consider the subsemigroup
	\[
		L_{GL(U)}(V)=\{\alpha\in L(V):\alpha|_U\in GL(U)\}
	\]
	of $L(V)$. In this work, Green's relations and ideals of this semigroup are described. Then we also determine the minimal ideal and the set of all minimal idempotents of it. Moreover, we establish an isomorphism theorem when $V$ is a finite dimensional vector space over a finite field. Finally, we find its generating set.
\end{abstract}
\noindent\textbf{2020 Mathematics Subject Classification:} 20M20, 20M17, 15A04, 15A03\\
\noindent\textbf{Keywords:} linear transformation semigroup, general linear group, Green's relations, ideal, generating set.


\section{Introduction}

Let $T(X)$ be the full transformation semigroup on a set $X$ under composition of functions. For a nonempty subset $Y$ of $X$, E. Laysirikul \cite{Laysirikul} defined a subsemigroup $PG_Y(X)$ of $T(X)$ by 
\[
	PG_Y(X)=\{\alpha\in T(X):\alpha|_Y\in G(Y)\},
\]
where $G(Y)$ is the symmetric group on $Y$ and $\alpha|_Y$ is the restriction of $\alpha$ to $Y$. The author demonstrated that $PG_Y(X)$ is regular and provided descriptions of left regularity, right regularity, and complete regularity of elements in $PG_Y(X)$. In 2021, W. Sommanee \cite{Sommanee} presented a characterization of Green's relations and described ideals of $PG_Y(X)$. Additionally, the author established some isomorphism theorems for $PG_Y(X)$. For a finite set $X$, he calculated the cardinalities of $PG_Y(X)$ and its subsets of idempotents, and also computed their ranks.

In 2022, J. Konieczny \cite{Konieczny} studied the properties of a semigroup denoted as $T_{\mathbb{S}(Y)}(X)$, where $Y$ is a nonempty subset of $X$ and $\mathbb{S}(Y)$ is a subsemigroup of $T(Y)$, the set of all transformations on $Y$. $T_{\mathbb{S}(Y)}(X)$ consists of transformations $\alpha\in T(X)$ such that the restriction of $\alpha$ to $Y$ belongs to $\mathbb{S}(Y)$, namely,
\[
	T_{\mathbb{S}(Y)}(X)=\{\alpha\in T(X):\alpha|_Y\in \mathbb{S}(Y)\}.
\]
In \cite{Konieczny}, the author focused on characterizing the regular elements of $T_{\mathbb{S}(Y)}(X)$ and investigated conditions under which $T_{\mathbb{S}(Y)}(X)$ forms a regular semigroup, inverse semigroup, or completely regular semigroup. Assuming that $\mathbb{S}(Y)$ includes the identity transformation $1_Y$, they characterized Green's relations of $T_{\mathbb{S}(Y)}(X)$. Then they applied these general results to derive specific outcomes for the semigroup $T_{\Gamma(Y)}(X)$, where $\Gamma(Y)$ represents the semigroup of full injective transformations on $Y$. Additionally, they explored generalizations and extensions of the semigroup $T_{\mathbb{S}(Y)}(X)$. Recently, M. Sarkar and S. N. Singh \cite{Sarkar} provided a new characterization for $T_{\mathbb{S}(Y)}(X)$ to be a regular semigroup, or inverse semigroup.

Let $X_{n}=\{1,\,\ldots ,\,n\}$ where $n\geq 2$. The symmetric group on $X_{n}$ is denoted by $S_{n}$, the symmetric inverse semigroup by $I_{n}$, the full transformation semigroup by $T_{n}$, and the partial transformation semigroup by $P_{n}$. In \cite{Bugay}, L. Bugay, R. Sönmez and H. Ay{\i}k determined the ranks of certain subsemigroups of $I_{n}$, $T_{n}$, and $P_{n}$ consisting of transformations whose restrictions to the set $X_{m}$ belong to the semigroup $S_{m}$, $I_{m}$, $T_{m}$, or $P_{m}$ for $1\leq m\leq n-1$.

Let $L(V)$ be the semigroup of all linear transformations on a vector space $V$. For a subspace $U$ of $V$ and a subsemigroup $\mathbb{S}(U)$ of $L(U)$, analogous to $T_{\mathbb{S}(Y)}(X)$, the authors in \cite{Sarkar} defined a subsemigroup 
\[
	L_{\mathbb{S}(U)}(V)=\{\alpha\in L(V):\alpha|_U\in \mathbb{S}(U)\}
\]
of $L(V)$. They described regular elements in $L_{\mathbb{S}(U)}(V)$ and investigated when $L_{\mathbb{S}(U)}(V)$ is a regular semigroup, inverse semigroup, or completely regular semigroup. If $\mathbb{S}(Y)$ (or $\mathbb{S}(U)$) includes the identity of $T(Y)$ (or $L(U)$), they described unit-regular elements in $T_{\mathbb{S}(Y)}(X)$ (or $L_{\mathbb{S}(U)}(V)$) and determined when $T_{\mathbb{S}(Y)}(X)$ (or $L_{\mathbb{S}(U)}(V)$) is a unit-regular semigroup.

We will use the notation $GL(V)$ as the set of all automorphisms of a vector space $V$, in simpler terms, the set of all bijective linear transformations from $V$ onto $V$. It is well-known that $GL(V)$ together with composition as operation is a group called the \textit{general linear group}.

In this paper, we study a special case of $L_{\mathbb{S}(U)}(V)$ when $\mathbb{S}(U)$ is the general linear group of $U$. In other words, let $U$ be a subspace of a vector space $V$ and define
\[
	L_{GL(U)}(V)=\{\alpha\in L(V):\alpha|_U\in GL(U)\}.
\]
If $U=V$, then $L_{GL(U)}(V)=GL(U)$. From now on, we assume that $U\neq V$. Moreover, we have $L_{GL(U)}(V)^1=L_{GL(U)}(V)$ since $L_{GL(U)}(V)$ contains the identity map $1_V$ on $V$ as the identity element. 

We organize the paper as follows. In Section \ref{sec: Preliminaries}, we provide some notations and results that will be used later. In Section \ref{sec: Green's relations and ideals}, we describe Green's relations and ideals of $L_{GL(U)}(V)$. At the end of this section, we determine the minimal ideal of $L_{GL(U)}(V)$. In Section \ref{sec: Isomorphism theorems}, we establish an isomorphism theorem for $L_{GL(U)}(V)$. Finally, in Section \ref{sec: Generating sets}, we find a generating set for $L_{GL(U)}(V)$.


\section{Preliminaries}\label{sec: Preliminaries}

We first state some notations and results of linear algebra and semigroup theory that will be used later. For all undefined notions, the reader is referred to \cite{Clifford1, Clifford2, Howie, Roman}.

In this paper, a subspace $U$ of a vector space $V$, spanned by a linearly independent subset $\{e_i\}$ of $V$, is denoted by $\langle e_i\rangle$. When we write $U =\langle e_i\rangle$, it means that the set $\{e_i\}$ serves as a basis for $U$, indicating that the dimension of $U$, denoted by $\dim U$, is $|I|$. It can be shown that $\dim U\leq\dim V$. Furthermore, if $\dim U=\dim V$ and $V$ is finite-dimensional, then $U=V$.

For each $\alpha\in L(V)$, the \textit{kernel} and \textit{range} of $\alpha$ are represented as $\ker\alpha$ and $V\alpha$, respectively. If a vector space $V$ is the internal direct sum of a family ${S_1,S_2,\ldots,S_n}$ of subspaces of $V$, it is denoted as $V=S_1\oplus S_2\oplus\cdots\oplus S_n$. Moreover, if $V=S\oplus T$, then $T$ is called a {\it complement} of $S$ in $V$. It should be emphasized that any subspace of a vector space has many complements, even though they are isomorphic.

Let $U$ be a subspace of a vector space $V$. The \textit{quotient space of $V$ modulo $U$}, denoted by $V/U$, is the set of all cosets of $U$ in $V$. It can be shown that all complements of $U$ in $V$ are isomorphic to $V/U$ and hence to each other. In particular, if $V$ is finite-dimensional, then $\dim(V/U)=\dim V-\dim U$.

Let $V$ be a vector space, and let $\{u_i\}$ be a subset of $V$. When we write $\sum a_iu_i$, it represents a finite linear combination:
$$
a_{i_1}u_{i_1}+a_{i_2}u_{i_2}+\cdots+a_{i_n}u_{i_n}
$$
where $n$ is a natural number, $u_{i_1},u_{i_2},\ldots,u_{i_n}$ are elements of $\{u_i\}$, and $a_{i_1},a_{i_2},\ldots,a_{i_n}$ are scalars. Suppose $\alpha\in L(V)$ and $U$ is a subspace of $V$. When we say $U\alpha=\langle u_j\alpha\rangle$, it means each $u_j$ belongs to $U$. It can be proven that $\{u_j\}$ is linearly independent. Furthermore, if $V\alpha=\langle v_i\alpha\rangle$, then $V$ can be decomposed into the direct sum of the kernel of $\alpha$ and the span of $\{v_i\}$, namely, $V=\ker\alpha\oplus\langle v_i\rangle$.

To simplify matters, we establish the following convention introduced in \cite[p. 241]{Clifford2}: if $\alpha$ is in $T(X)$, then we will represent it as 
$$
\alpha=\binom{A_i}{a_i}.
$$
It is implied that the subscript $i$ refers to an unspecified index set $I$, the abbreviation $\{a_i\}$ denotes $\{a_i: i\in I\}$, and that $\im\alpha=X\alpha = \{a_i\}$ and $a_i\alpha^{-1}= A_i$.

Likewise, we can utilize the above notation for elements within $L(V)$. To create a linear transformation $\alpha\in L(V)$, we initially select a basis $\{e_i\}$ for $V$ and a subset $\{a_i\}$ of $V$. Then we define $e_i\alpha = a_i$ for each $i\in I$ and extend this mapping linearly over $V$. To simplify this procedure, we can simply state that, given $\{e_i\}$ and $\{a_i\}$ in the given context, for any $\alpha\in L(V)$, we can represent it as
$$
\alpha=\binom{e_i}{a_i}.
$$

This paper relies on well-known results in linear algebra, which are stated below.

\begin{theorem}
	Let $V$ and $W$ be vector spaces over a field $\mathbb{F}$. Then $V\cong W$ if and only if $\dim V=\dim W$.
\end{theorem} 

\begin{theorem}
	Let $S$ and $T$ be subspaces of a vector space $V$ such that $S\cap T=\{0\}$. Then $\dim(S\oplus T)=\dim S+\dim T$.
\end{theorem}


\section{Green's relations and ideals}\label{sec: Green's relations and ideals}

In this section, we will investigate Green's relations and ideals of $L_{GL(U)}(V)$. At the end of this section, we will describe the minimal ideal of $L_{GL(U)}(V)$.

We first consider the regularity of $L_{GL(U)}(V)$. To do this, the proposition appeared in \cite{Sarkar} is needed.

\begin{proposition}\cite[Proposition 4.3]{Sarkar}
	If $\mathbb{S}(U)$ is a subgroup of $GL(U)$, then $L_{\mathbb{S}(U)}(V)$ is regular.
\end{proposition}

By a direct consequence of the above proposition, we have the following result immediately.

\begin{theorem}
	$L_{GL(U)}(V)$ is a regular semigroup.
\end{theorem}

The following lemma will be useful for proving a characterization of Green's relations and ideals.

\begin{lemma}\label{lem: pre ideal}
	Let $\alpha,\beta\in L_{GL(U)}(V)$. Then $\dim(V\alpha/U)\leq\dim(V\beta/U)$ if and only if $\alpha=\lambda\beta\mu$ for some $\lambda,\mu\in L_{GL(U)}(V)$.
\end{lemma}
\begin{proof}
	Let $\alpha,\beta\in L_{GL(U)}(V)$ be such that $V\alpha=W_1\oplus U$ and $V\beta=W_2\oplus U$. Suppose that $\dim(V\alpha/U)\leq\dim(V\beta/U)$. Then $\dim W_1\leq\dim W_2$. We can write $\ker\alpha=\langle v_r\rangle$, $\ker\beta=\langle v_s\rangle$, $W_1=\langle w_i\alpha\rangle$, $W_2=\langle w'_i\beta\rangle\oplus\langle w_j\beta\rangle$ and $U=\langle u_k\rangle$. We have $\langle u_k\alpha\rangle=U=\langle u_k\beta\rangle$ since $\alpha,\beta\in L_{GL(U)}(V)$. Then
	\[
		\alpha= \left(
		\begin{array}{ccc}
		v_r & w_i       & u_k\\
		0   & w_i\alpha & u_k\alpha\\ 
		\end{array}
		\right)\ \text{and}\ 
		\beta= \left(
		\begin{array}{cccc}
		v_s & w'_i      & w_j      & u_k\\
		0   & w'_i\beta & w_j\beta & u_k\beta\\ 
		\end{array}
		\right).
	\]
	Let $V=\langle w'_i\beta\rangle\oplus\langle u_k\beta\rangle\oplus\langle v_t\rangle$ and define
	\[
		\lambda= \left(
		\begin{array}{ccc}
		v_r & w_i  & u_k\\
		0   & w'_i & u_k\\ 
		\end{array}
		\right)\ \text{and}\ 
		\mu= \left(
		\begin{array}{ccc}
		v_t & w'_i\beta & u_k\beta\\
		0   & w_i\alpha & u_k\alpha\\ 
		\end{array}
		\right).
	\]
	It is straightforward to verify that $\lambda,\mu\in L_{GL(U)}(V)$ and $\alpha=\lambda\beta\mu$.

	Conversely, assume that $\alpha=\lambda\beta\mu$ for some $\lambda,\mu\in L_{GL(U)}(V)$. Let $U=\langle u_k\rangle$. Then we can write $V\alpha=\langle w_i\alpha\rangle\oplus\langle u_k\alpha\rangle$. Moreover, $V\alpha=V\lambda\beta\mu\subseteq V\mu$ which implies that $V\mu=A\oplus\langle w_i\alpha\rangle\oplus\langle u_k\alpha\rangle=A\oplus\langle w_i\lambda\beta\mu\rangle\oplus\langle u_k\lambda\beta\mu\rangle$ for some complement $A$ in $V\mu$. Hence $V\beta=B\oplus\langle w_i\lambda\beta\rangle\oplus\langle u_k\lambda\beta\rangle=B\oplus\langle w_i\lambda\beta\rangle\oplus U$ for some complement $B$ in $V\beta$. Therefore,
	\begin{eqnarray*}
		\dim(V\beta/U)&=&\dim(B\oplus\langle w_i\lambda\beta\rangle)=\dim B+\dim\langle w_i\lambda\beta\rangle\\
		&\geq&\dim\langle w_i\lambda\beta\rangle=\dim\langle w_i\alpha\rangle=\dim(V\alpha/U).
	\end{eqnarray*}
\end{proof}

\begin{corollary}\label{cor: dim inequality}
	Let $\alpha,\beta\in L_{GL(U)}(V)$. Then
	\[
		\dim(V\alpha\beta/U)\leq\dim(V\alpha/U)\ \text{and}\ \dim(V\alpha\beta/U)\leq\dim(V\beta/U).
	\]
\end{corollary}
\begin{proof}
	We see that $\alpha\beta=1_V\alpha\beta$ which implies by Lemma \ref{lem: pre ideal} that $\dim(V\alpha\beta/U)\leq\dim(V\alpha/U)$. The inequality $\dim(V\alpha\beta/U)\leq\dim(V\beta/U)$ can be shown similarly.
\end{proof}

To consider Green's relations on $L_{GL(U)}(V)$, we adopt the following notations introduced by \cite{Howie}. If $T$ is a subsemigroup of a semigroup $S$ and $a,b\in T$, then $(a,b)\in\mathscr{L}^T$ means that there exist $u,v\in T^1$ such that $ua=b$, $vb=a$ while $(a,b)\in\mathscr{L}^S$ means that there exist $s,t\in S^1$ such that $sa=b,$ $tb=a$. Obviously,
\[
	\mathscr{L}^T\subseteq\mathscr{L}^S\cap(T\times T).
\]
Similarly, we also use the following notations:
\begin{center}{$\mathscr{R}^T\subseteq\mathscr{R}^S\cap(T\times T)$,~~~~~$\mathscr{H}^T\subseteq\mathscr{H}^S\cap(T\times T),$}\end{center}
\begin{center}{$\mathscr{D}^T\subseteq\mathscr{D}^S\cap(T\times T)$,~~~~~$\mathscr{J}^T\subseteq\mathscr{J}^S\cap(T\times T).$}\end{center}
Due to the results given by T. E. Hall, it is well-known that if $T$ is a regular subsemigroup of $S$, we have
\begin{center}{$\mathscr{L}^T=\mathscr{L}^S\cap(T\times T)$,~~~~~$\mathscr{R}^T=\mathscr{R}^S\cap(T\times T)$,~~~~~$\mathscr{H}^T=\mathscr{H}^S\cap(T\times T)$}\end{center}
(see Proposition 2.4.2 of \cite{Howie} for details).

\begin{sloppypar}
	For convenience, we write $\mathscr{L}^U=\mathscr{L}^{L_{GL(U)}(V)}$. The same notation applies to $\mathscr{R}^U,\mathscr{H}^U,\mathscr{D}^U$ and $\mathscr{J}^U$.
\end{sloppypar}

Refer to the characterizations of Green's relations on $L(V)$ in \cite{Howie}, the authors obtained the following results. For each $\alpha,\beta\in L(V)$,
\begin{enumerate}[(1)]
	\item $(\alpha,\beta)\in\mathscr{L}^{L(V)}$ if and only if $V\alpha=V\beta$;
	\item $(\alpha,\beta)\in\mathscr{R}^{L(V)}$ if and only if $\ker\alpha=\ker\beta$;
	\item $(\alpha,\beta)\in\mathscr{H}^{L(V)}$ if and only if $V\alpha=V\beta$ and $\ker\alpha=\ker\beta$;
	\item $(\alpha,\beta)\in\mathscr{D}^{L(V)}$ if and only if $\dim(V\alpha)=\dim(V\beta)$;
	\item $\mathscr{D}^{L(V)}=\mathscr{J}^{L(V)}$.
\end{enumerate}

Now, we are in position to describe Green's relations on $L_{GL(U)}(V)$.

\begin{theorem}\label{thm: Green's relations on LGL(U)}
	Let $\alpha,\beta\in L_{GL(U)}(V)$. Then the following statements hold.
	\begin{enumerate}[(1)]
		\item $(\alpha,\beta)\in\mathscr{L}^U$ if and only if $V\alpha=V\beta$.
		\item $(\alpha,\beta)\in\mathscr{R}^U$ if and only if $\ker\alpha=\ker\beta$.
		\item $(\alpha,\beta)\in\mathscr{H}^U$ if and only if $V\alpha=V\beta$ and $\ker\alpha=\ker\beta$
		\item $(\alpha,\beta)\in\mathscr{D}^U$ if and only if $\dim(V\alpha/U)=\dim(V\beta/U)$.
		\item $\mathcal{D}^U=\mathcal{J}^U$.
	\end{enumerate}
\end{theorem}
\begin{proof}
	The statements (1), (2) and (3) follow from \cite[Proposition 2.4.2]{Howie}.

	To show (4), suppose that $(\alpha,\beta)\in\mathscr{D}^U$. Then $(\alpha,\gamma)\in\mathscr{L}^U$ and $(\gamma,\beta)\in\mathscr{R}^U$ for some $\gamma\in L_{GL(U)}(V)$. It follows by (1) and (2) that $V\alpha=V\gamma$ and $\ker\gamma=\ker\beta$. By using the first isomorphism theorem, we obtain
	\[
		\dim(V\alpha/U)=\dim(V\gamma/U)=\dim[(V/\ker\gamma)/U]=\dim[(V/\ker\beta)/U]=\dim(V\beta/U).
	\]

	Conversely, assume that $\dim(V\alpha/U)=\dim(V\beta/U)$. Let $U=\langle u_k\rangle$, $\ker\alpha=\langle v_i\rangle$ and $\ker\beta=\langle v_s\rangle$. We can write
	\[
		\alpha= \left(
		\begin{array}{ccc}
		v_i & v_j       & u_k\\
		0   & v_j\alpha & u_k\alpha\\ 
		\end{array}
		\right)\ \text{and}\ 
		\beta= \left(
		\begin{array}{ccc}
		v_s & v_j'      & u_k \\
		0   & v_j'\beta & u_k\beta \\ 
		\end{array}
		\right)
	\]
	where $\langle v_j\alpha\rangle$ and $\langle v_j'\beta\rangle$ are complements of $U$ in $V\alpha$ and $V\beta$, respectively. Define a function $\gamma$ by
	\[
		\gamma= \left(
		\begin{array}{ccc}
		v_s & v_j'      & u_k\\
		0   & v_j\alpha & u_k\alpha\\ 
		\end{array}
		\right).
	\]
	It is straightforward to verify that $\gamma\in L_{GL(U)}(V)$ such that $V\alpha=V\gamma$ and $\ker\gamma=\ker\beta$. Hence $\alpha\mathcal{L}^U\gamma\mathcal{R}^U\beta$ and thus $(\alpha,\beta)\in\mathscr{D}^U$.

	The statement (5) is a direct consequence of Lemma \ref{lem: pre ideal}.
\end{proof}

Now, we turn our attention to the ideals of $L_{GL(U)}(V)$.

\begin{theorem}
	The proper ideals of $L_{GL(U)}(V)$ are precisely the sets
	\[
		Q(k)=\{\alpha\in L_{GL(U)}(V):\dim(V\alpha/U)<k\}
	\]
	where $1\leq k\leq\dim(V/U)$.
\end{theorem}
\begin{proof}
	Let $\alpha\in L_{GL(U)}(V)$ and $\beta\in Q(k)$. By Corollary \ref{cor: dim inequality}, it is straightforward to verify that $\alpha\beta$ and $\beta\alpha$ are in $Q(k)$. Hence $Q(k)$ is an ideal. We can see that the identity map $1_V$ on $V$ is in $L_{GL(U)}(V)$ but not in $Q(k)$ since $\dim(V1_V/U)=\dim(V/U)\geq k$. Therefore, $Q(k)$ is a proper ideal of $L_{GL(U)}(V)$.

	Let $I$ be an ideal of $L_{GL(U)}(V)$ and $k$ the least cardinal greater than $\dim(V\alpha/U)$ for all $\alpha\in I$. Then $1\leq k\leq\dim(V/U)$. Obviously, $I\subseteq Q(k)$. We now focus on the other containment. Let $\alpha\in Q(k)$. Then $\dim(V\alpha/U)<k$. If $\dim(V\beta/U)<\dim(V\alpha/U)<k$ for all $\beta\in I$, then it contradicts to the choice of $k$. Hence there is $\beta\in I$ such that $\dim(V\alpha/U)\leq\dim(V\beta/U)$. It follows by Lemma \ref{lem: pre ideal} that $\alpha=\lambda\beta\mu$ for some $\lambda,\mu\in L_{GL(U)}(V)$. Thus $\alpha\in I$ since $I$ is an ideal.
\end{proof}

For each $1\leq k\leq\dim(V/U)$, define a subset $J(k)$ of $L_{GL(U)}(V)$ by
\[
	J(k)=\{\alpha\in L_{GL(U)}(V):\dim(V\alpha/U)=k\}.
\]
It follows by Theorem \ref{thm: Green's relations on LGL(U)} that $J(k)$ is a $\mathcal{J}$-class of $L_{GL(U)}(V)$. We can see that the number of $\mathcal{J}$-classes of $L_{GL(U)}(V)$ is $\dim(V/U)$. Obviously, if $\dim V$ is finite, then
\[
	Q(k)=J(0)\dot\cup J(1)\dot\cup J(2)\dot\cup\cdots\dot\cup J(k-1).
\]

It is clear that $Q(1)\subseteq Q(2)\subseteq\cdots$ and hence it forms a chain under inclusion. Consequently, we obtain the following result immediately.

\begin{corollary}
	The minimal ideal of $L_{GL(U)}(V)$ is the set
	\[
		Q(1)=J(0)=\{\alpha\in L_{GL(U)}(V):\dim(V\alpha/U)=0\}.
	\]
\end{corollary}

It is straightforward to verify that
\[
	Q(1)=\{\alpha\in L_{GL(U)}(V):V\alpha=U\}=\{\alpha\in L_{GL(U)}(V):V=\ker\alpha\oplus U\}.
\]


\section{Isomorphism theorems}\label{sec: Isomorphism theorems}

Let $S$ be a semigroup. We denote the set of all idempotents in $S$ by $E(S)$. Recall that the \textit{natural partial order} $\leq$ on $E(S)$ defined by, for $e, f\in E(S)$, 
\[
	e\leq f\ \text{if and only if}\ e=ef=fe.
\]
A \textit{minimal idempotent} is an idempotent $e\in E(S)$ such that there is no idempotent $f\in E(S)$ with $f<e$. In order to establish an isomorphism theorem for $L_{GL(U)}(V)$, we will first describe the set of all minimal idempotents of $L_{GL(U)}(V)$.

\begin{lemma}\label{lem: idempotent}
	$\epsilon\in L_{GL(U)}(V)$ is an idempotent if and only if $\epsilon|_{V\epsilon}$ is the identity map $1_{V\epsilon}$.
\end{lemma}
\begin{proof}
	Suppose that $\epsilon\in L_{GL(U)}(V)$ is an idempotent. Let $w\in V\epsilon$. Then $w=v\epsilon$ for some $v\in V$. We obtain $w\epsilon=v\epsilon^2=v\epsilon=w$. Conversely, assume that $\epsilon|_{V\epsilon}$ is the identity map $1_{V\epsilon}$. Let $v\in V$. Then $v\epsilon\in V\epsilon$ which implies that $v\epsilon^2=(v\epsilon)\epsilon=v\epsilon$. Thus $\epsilon$ is an idempotent.
\end{proof}

By the above lemma, the set of all idempotents in $L_{GL(U)}(V)$ is $E(L_{GL(U)}(V))=\{\epsilon\in L_{GL(U)}(V):\epsilon|_{V\epsilon}=1_{V\epsilon}\}$.

\begin{theorem}
	Let $M=\{\alpha\in E(L_{GL(U)}(V)):V\alpha=U\}$. Then $M$ is the set of all minimal idempotents in $L_{GL(U)}(V)$.
\end{theorem}
\begin{proof}
	Let $\alpha\in M$ and $\beta\in E(L_{GL(U)}(V))$ be such that $\beta\leq\alpha$. Then $\beta=\beta\alpha=\alpha\beta$. To show that $\alpha=\beta$, let $v\in V$. Then $v\beta=v\alpha\beta=v\alpha$ since $v\alpha\in U\subseteq V\beta$ and $\beta$ is idempotent. Hence $\alpha$ is minimal.

	Conversely, let $\alpha$ be a minimal idempotent. By Lemma \ref{lem: idempotent}, we can write
	\[
		\alpha= \left(
		\begin{array}{ccc}
		v_i & v_j & u_k\\
		0   & v_j & u_k\\ 
		\end{array}
		\right)
	\]
	where $\ker\alpha=\langle v_i\rangle$, $V\alpha=\langle v_j\rangle\oplus U=\langle v_j\rangle\oplus\langle u_k\rangle$. Define an idempotent $\epsilon$ by
	\[
		\epsilon= \left(
		\begin{array}{ccc}
		v_i & v_j & u_k\\
		0   & 0   & u_k\\ 
		\end{array}
		\right).
	\]
	It is straightforward to verify that $\epsilon=\epsilon\alpha=\alpha\epsilon$. Hence $\epsilon\leq\alpha$ and so $\alpha=\epsilon$ since $\alpha$ is minimal. We have $V\alpha=V\epsilon=U$. Therefore, $\alpha\in M$.
\end{proof}

By the above proposition, if $\epsilon$ is a minimal idempotent in $L_{GL(U)}(V)$, we can write
\[
	\epsilon= \left(
	\begin{array}{cc}
	v_i & u_k\\
	0   & u_k\\ 
	\end{array}
	\right)
\]
where $\ker\epsilon=\langle v_i\rangle$ and $U=\langle u_k\rangle$. Obviously, the minimal ideal $Q(1)$ contains the set of all minimal idempotents $M$. Moreover, we can see that the number of elements in $M$ is equal to the possibility of the kernel of element which is a complement of $U$ in $V$. It follows that the cardinality of $M$ is the number of all complements of $U$ in $V$.

Let $V_1$ and $V_2$ be vector spaces over a field $\mathbb{F}$ and let $U_1$ and $U_2$ be subspaces of $V_1$ and $V_2$, respectively. If $L_{GL(U_1)}(V_1)$ and $L_{GL(U_2)}(V_2)$ are isomorphic, then the cardinality of the set of all minimal idempotents of these two semigroups are equal. Hence the number of all complements of $U_1$ in $V_1$ and the number of all complements of $U_2$ in $V_2$ are the same. 

For the finite case, according to \cite[Theorem 6]{Tingley}, the author counted the number of complements of a subspace in a vector space as follows.

\begin{theorem}\cite[Theorem 6]{Tingley}\label{thm: the number of complements}
	Let $V$ be an $n$-dimensional vector space over a finite field $\mathbb{F}$. Let $U$ be any $k$-dimensional subspace of $V$. Then there are $|\mathbb{F}|^{k(n-k)}$ distinct complements of $U$ in $V$.
\end{theorem}

Using the above theorem, we are in position to prove an isomorphism theorem for the finite case.

\begin{theorem}
	Let $V_1$ and $V_2$ be finite dimensional vector spaces over a finite field $\mathbb{F}$ and let $U_1$ and $U_2$ be subspaces of $V_1$ and $V_2$, respectively. Then $L_{GL(U_1)}(V_1)$ and $L_{GL(U_2)}(V_2)$ are isomorphic if and only if there is an isomorphism $\phi:V_1\to V_2$ such that $U_1\phi=U_2$.
\end{theorem}
\begin{proof}
	Assume that $L_{GL(U_1)}(V_1)$ and $L_{GL(U_2)}(V_2)$ are isomorphic. As we mentioned before, the number of all complements of $U_1$ in $V_1$ and the number of all complements of $U_2$ in $V_2$ are the same. Let $\dim V_1=m$, $\dim V_2=n$, $\dim U_1=p$ and $\dim U_2=q$. By Theorem \ref{thm: the number of complements}, we have $|\mathbb{F}|^{p(m-p)}=|\mathbb{F}|^{q(n-q)}$. Furthermore, the number of $\mathcal{J}$-classes of $L_{GL(U_1)}(V_1)$ and $L_{GL(U_2)}(V_2)$ are $\dim(V_1/U_1)=m-p$ and $\dim(V_2/U_2)=n-q$, respectively. Hence $m-p=n-q$ since these two semigroups are isomorphic. Thus $p=q$ and $m=n$. We can write $V_1=\langle v_1,v_2,\ldots,v_{m-p}\rangle\oplus\langle u_1,u_2,\ldots,u_p\rangle$ and $V_2=\langle v'_1,v'_2,\ldots,v'_{m-p}\rangle\oplus\langle u'_1,u'_2,\ldots,u'_p\rangle$ where $U_1=\langle u_1,u_2,\ldots,u_p\rangle$ and $U_2=\langle u'_1,u'_2,\ldots,u'_p\rangle$. Define a linear transformation $\phi: V_1\to V_2$ by
	\[
		\phi= \left(
		\begin{array}{cccccccc}
		v_1 & v_2 & \cdots & v_{m-p} & u_1 & u_2 & \cdots & u_p\\
		v'_1 & v'_2 & \cdots & v'_{m-p} & u'_1 & u'_2 & \cdots & u'_p\\ 
		\end{array}
		\right).
	\]
	It is routine matter to verify that $\phi$ is an isomorphism such that $U_1\phi=U_2$.

	Conversely, suppose there is an isomorphism $\phi:V_1\to V_2$ such that $U_1\phi=U_2$. Define a function $\Psi:L_{GL(U_1)}(V_1)\to L_{GL(U_2)}(V_2)$ by
	\[
		\alpha\Psi=\phi^{-1}\alpha\phi\ \text{for all}\ \alpha\in L_{GL(U_1)}(V_1).
	\]
	It is straightforward to show that $\Psi$ is an isomorphism.
\end{proof}


\section{Generating sets}\label{sec: Generating sets}

Recall that a \textit{generating set} of a semigroup $S$ is a subset $G$ of $S$ such that every element in $S$ can be expressed as a finite product of elements from $G$. In other words, if we take any element in the semigroup $S$, we can write it as a combination of elements from the generating set $G$, using the operation of the semigroup. If $G$ is a generating set of $S$, then we write $S=\langle G\rangle$. For any semigroup $S$, the \textit{rank} of $S$, denoted by $\rank(S)$, is the smallest cardinality of a generating set of $S$. In other words, 
\[
	\rank(S)=\min\{|G|:G\subseteq S\ \text{and}\ S=\langle G\rangle\}.
\] 
In this section, we will describe a generating set and rank of $L_{GL(U)}(V)$.

In what follows, we assume that $V$ is an $n$-dimensional vector space over a finite field $\mathbb{F}$ and $U$ is an $r$-dimensional subspace of $V$. From now on, we write
\[
	U=\langle u_1,u_2,\ldots,u_r\rangle.
\]

To find a generating set of $L_{GL(U)}(V)$, the following lemmas will be useful.

\begin{lemma}\label{lem: J(k) generated by J(k+1)}
	Let $0\leq k\leq n-r-2$ and $\alpha\in J(k)$. Then $\alpha=\lambda\mu$ for some $\lambda,\mu\in J(k+1)$. 
\end{lemma}
\begin{proof}
	We can write
	\[
		\alpha= \left(
		\begin{array}{cccccccccccc}
		v_1 & v_2 & \cdots & v_k & v_{k+1} & v_{k+2} & \cdots & v_{n-r} & u_1 & u_2 & \cdots & u_r\\
		v_1\alpha & v_2\alpha & \cdots & v_k\alpha & 0 & 0 & \cdots & 0 & u_1\alpha & u_2\alpha & \cdots & u_r\alpha\\ 
		\end{array}
		\right)
	\]
	where $\ker\alpha=\langle v_{k+1},v_{k+2},\ldots,v_{n-r}\rangle$ and $V\alpha=\langle v_1\alpha,v_2\alpha,\ldots,v_k\alpha\rangle\oplus U$. Let
	\[
		V=\langle v_1\alpha,v_2\alpha,\ldots,v_k\alpha\rangle\oplus U\oplus\langle w_1,w_2,\ldots,w_{n-r-k}\rangle.
	\]
	We note that $\dim(\ker\alpha)=\dim\langle w_1,w_2,\ldots,w_{n-r-k}\rangle\geq 2$ since $k\leq n-r-2$. Define $\lambda,\mu\in L_{GL(U)}(V)$ by
	\[
		\lambda= \left(
		\begin{array}{cccccccccccc}
		v_1 & v_2 & \cdots & v_k & v_{k+1} & v_{k+2} & \cdots & v_{n-r} & u_1 & u_2 & \cdots & u_r\\
		v_1\alpha & v_2\alpha & \cdots & v_k\alpha & w_1 & 0 & \cdots & 0 & u_1\alpha & u_2\alpha & \cdots & u_r\alpha\\ 
		\end{array}
		\right)
	\]
	and
	\[
		\mu= \left(
		\begin{array}{ccccccccccccc}
		v_1\alpha & v_2\alpha & \cdots & v_k\alpha & w_1 & w_2 & w_3 & \cdots & w_{n-r-k} & u_1\alpha & u_2\alpha & \cdots & u_r\alpha\\
		v_1\alpha & v_2\alpha & \cdots & v_k\alpha & 0 & w_2 & 0 & \cdots & 0 & u_1\alpha & u_2\alpha & \cdots & u_r\alpha\\ 
		\end{array}
		\right).
	\]
	It is easy to see that $\lambda,\mu\in J(k+1)$ and $\alpha=\lambda\mu$.
\end{proof}

Inductive application of Lemma \ref{lem: J(k) generated by J(k+1)} yields the following corollary.

\begin{corollary}\label{cor: Q(k+1)=<J(k)>}
	$Q(k+1)=\langle J(k)\rangle$ for all $1\leq k\leq n-r-1$.
\end{corollary}

\begin{lemma}\label{lem: J(n-r-1) subset J(n-r) alpha J(n-r)}
	$J(n-r-1)\subseteq J(n-r)\alpha J(n-r)$ for all $\alpha\in J(n-r-1)$
\end{lemma}
\begin{proof}
	Let $\alpha\in J(n-r-1)$. Let $V\alpha=\langle v_1\alpha,v_2\alpha,\ldots,v_{n-r-1}\alpha\rangle\oplus U$ and $\ker\alpha=\langle v_{n-r}\rangle$. We can write
	\[
		\alpha= \left(
		\begin{array}{ccccccccc}
		v_1 & v_2 & \cdots & v_{n-r-1} & v_{n-r} & u_1 & u_2 & \cdots & u_r\\
		v_1\alpha & v_2\alpha & \cdots & v_{n-r-1}\alpha & 0 & u_1\alpha & u_2\alpha & \cdots & u_r\alpha\\ 
		\end{array}
		\right).
	\]
	Let $\beta\in J(n-r-1)$. Similar to $\alpha$, we can write
	\[
		\beta= \left(
		\begin{array}{ccccccccc}
		w_1 & w_2 & \cdots & w_{n-r-1} & w_{n-r} & u_1 & u_2 & \cdots & u_r\\
		w_1\beta & w_2\beta & \cdots & w_{n-r-1}\beta & 0 & u_1\beta & u_2\beta & \cdots & u_r\beta\\ 
		\end{array}
		\right).
	\]
	Let $V=\langle v_1\alpha,v_2\alpha,\ldots,v_{n-r-1}\alpha\rangle\oplus\langle w\rangle\oplus U=\langle w_1\beta,w_2\beta,\ldots,w_{n-r-1}\beta\rangle\oplus\langle w'\rangle\oplus U$. Define $\lambda,\mu\in L_{GL(U)}(V)$ by
	\[
		\lambda= \left(
		\begin{array}{ccccccccc}
		w_1 & w_2 & \cdots & w_{n-r-1} & w_{n-r} & u_1 & u_2 & \cdots & u_r\\
		v_1 & v_2 & \cdots & v_{n-r-1} & v_{n-r} & u_1 & u_2 & \cdots & u_r\\ 
		\end{array}
		\right)
	\]
	and
	\[
		\mu= \left(
		\begin{array}{ccccccccc}
		v_1\alpha & v_2\alpha & \cdots & v_{n-r-1}\alpha & w & u_1\alpha & u_2\alpha & \cdots & u_r\alpha\\
		w_1\beta & w_2\beta & \cdots & w_{n-r-1}\beta & w' & u_1\beta & u_2\beta & \cdots & u_r\beta\\ 
		\end{array}
		\right).
	\]
	It is straightforward to verify that $\beta=\lambda\alpha\mu$ and $\lambda,\mu\in J(n-r)$.
\end{proof}

\begin{lemma}\label{lem: upper bound rank}
	Let $\alpha\in J(n-r-1)$. Then $L_{GL(U)}(V)=\langle J(n-r)\cup\{\alpha\}\rangle$. Consequently, $\rank(L_{GL(U)}(V))\leq\rank(J(n-r))+1$.
\end{lemma}
\begin{proof}
	By Corollary \ref{cor: Q(k+1)=<J(k)>} and Lemma \ref{lem: J(n-r-1) subset J(n-r) alpha J(n-r)}, we have $J(n-r-1)\subseteq J(n-r)\alpha J(n-r)\subseteq\langle J(n-r)\cup\{\alpha\}\rangle$ and $Q(n-r)=\langle J(n-r-1)\rangle\subseteq\langle J(n-r)\cup\{\alpha\}\rangle$. Thus
	\[
		L_{GL(U)}(V)=Q(n-r)\cup J(n-r)\subseteq\langle J(n-r)\cup\{\alpha\}\rangle\cup J(n-r)=\langle J(n-r)\cup\{\alpha\}\rangle.
	\]
	Therefore, $L_{GL(U)}(V)=\langle J(n-r)\cup\{\alpha\}\rangle$.
\end{proof}

We have $\langle J(n-r)\rangle=J(n-r)\neq L_{GL(U)}(V)$ and an element in $J(n-r)$ cannot be written as a product of some elements in $Q(n-r)$ since $Q(n-r)$ is an ideal. Thus $\rank(L_{GL(U)}(V))\geq\rank(J(n-r))+1$. Therefore, by Lemma \ref{lem: upper bound rank}, we obtain the following result immediately.

\begin{corollary}
	$\rank(L_{GL(U)}(V))=\rank(J(n-r))+1$.
\end{corollary}

To determine the rank of $L_{GL(U)}(V)$, the rank of $J(n-r)$ is needed. In the remaining part of this paper, we will find a generating set of $J(n-r)$ and collecting some properties of it. 

\begin{remark}
	For each $\alpha\in L_{GL(U)}(V)$, it is routine matter to verify that $\alpha\in J(n-r)$ if and only if $V\alpha=V$. Furthermore, $\alpha\in L_{GL(U)}(V)$ is an element of $J(n-r)$ if and only if it is an automorphism of $V$.
\end{remark}

Now, we first show that $J(n-r)$ is a subgroup of $L_{GL(U)}(V)$.

\begin{theorem}
	$J(n-r)$ is a subgroup of $L_{GL(U)}(V)$.
\end{theorem}
\begin{proof}
	By the above remark, we have $J(n-r)\subseteq GL(V)$. Hence it suffices to show that $J(n-r)$ is a subgroup of a finite group $GL(V)$. Let $\alpha,\beta\in J(n-r)$. Then $V\alpha$ and $V\beta$ are $V$ which implies that $V\alpha\beta=V\beta=V$. Hence $\alpha\beta\in J(n-r)$ and so $J(n-r)$ is closed under composition.
\end{proof}

Recall that an \textit{internal semidirect product} of two groups $N$ and $H$, denoted by $H\ltimes N$, is a group $G$ such that $N$ is a normal subgroup of $G$, $N\cap H=\{1\}$, and $G=HN$. Next, we will show that $J(n-r)$ can be decomposed into an internal semidirect product of two subgroups.

\begin{lemma}
	Let $W$ be a subspace of $V$. Then the set
	\[
		\fix(W)=\{\alpha\in J(n-r):\alpha|_W=1_W\}
	\]
	is a subgroup of $J(n-r)$. Furthermore, $\fix(U)$ is a normal subgroup of $J(n-r)$.
\end{lemma}
\begin{proof}
	Obviously, the identity map $1_V$ is the identity element in $\fix(W)$. Let $\alpha,\beta\in\fix(W)$. Then $\alpha|_W=1_W=\beta|_W$ and so $(\alpha\beta)|_W=\alpha|_W\beta|_W=1_W$. Hence $\alpha\beta\in\fix(W)$.

	To prove normality of $\fix(U)$, let $\alpha\in\fix(U)$ and $\beta\in J(n-r)$. Then $\alpha|_U=1_U$ and so $(\beta\alpha\beta^{-1})|_U=\beta|_U\alpha|_U\beta^{-1}|_U=\beta|_U\beta^{-1}|_U=1_U$. Hence $\beta\alpha\beta^{-1}\in\fix(U)$.
\end{proof}

In general, if $W$ is a complement of $U$ in $V$, then the subgroup $\fix(W)$ as defined above is not normal in $J(n-r)$. See the following example.

\begin{example}
	Let $V=W\oplus U$ be such that $W=\{w\}$ and $U=\{u_1,u_2\}$. Let $\alpha\in J(n-r)=J(1)$ and $\beta\in\fix(W)$ be defined by
	\[
		\alpha= \left(
		\begin{array}{ccc}
		w   & u_1 & u_2\\
		w+u_1 & u_1 & u_2\\ 
		\end{array}
		\right)
	\]
	and
	\[
		\beta= \left(
		\begin{array}{ccc}
		w & u_1 & u_2\\
		w & u_2 & u_1\\ 
		\end{array}
		\right).
	\]
	Then $w\alpha\beta\alpha^{-1}=(w+u_1)\beta\alpha^{-1}=(w\beta+u_1\beta)\alpha^{-1}=(w+u_2)\alpha^{-1}=w\alpha^{-1}+u_2\alpha^{-1}=(w+u_1-u_1)\alpha^{-1}+u_2=(w+u_1)\alpha^{-1}-u_1\alpha^{-1}+u_2=w-u_1+u_2$ from which it follows that $W\alpha\beta\alpha^{-1}=\langle w-u_1+u_2\rangle\neq W$. Hence $\alpha\beta\alpha^{-1}\notin\fix(W)$. Therefore, $\fix(W)$ is not normal in $J(n-r)$.
\end{example}

The following theorem asserts that $J(n-r)$ can be decomposed into the internal semidirect product of two subgroups $\fix(U)$ and $\fix(W)$ where $W$ is a complement of $U$ in $V$.

\begin{theorem}
	Let $W$ be a complement of $U$ in $V$. Then $J(n-r)=\fix(W)\ltimes\fix(U)$.
\end{theorem}
\begin{proof}
	For each $\alpha\in J(n-r)$, we can write $V\alpha=\langle w_1\alpha,w_2\alpha,\ldots,w_{n-r}\alpha\rangle\oplus U$ where $W=\langle w_1,w_2,\ldots w_{n-r}\rangle$. Let 
	\[
		\alpha'=\left(
			\begin{array}{cccccccc}
			w_1 & w_2 & \cdots & w_{n-r} & u_1 & u_2 & \cdots & u_r\\
			w_1 & w_2 & \cdots & w_{n-r} & u_1\alpha & u_2\alpha & \cdots & u_r\alpha\\ 
			\end{array}
			\right)
	\]
	and
	\[
		\alpha''=\left(
			\begin{array}{cccccccc}
			w_1 & w_2 & \cdots & w_{n-r} & u_1\alpha & u_2\alpha & \cdots & u_r\alpha\\
			w_1\alpha & w_2\alpha & \cdots & w_{n-r}\alpha & u_1\alpha & u_2\alpha & \cdots & u_r\alpha\\ 
			\end{array}
			\right).
	\]
	It is clear that $\alpha'\in\fix(W)$, $\alpha''\in\fix(U)$ and $\alpha=\alpha'\alpha''$. Hence $J(n-r)=\fix(W)\fix(U)$. Moreover, it is easy to see that $\fix(U)\cap\fix(W)=\{1_V\}$. Therefore, $J(n-r)=\fix(W)\ltimes\fix(U)$ since $\fix(U)$ is a normal subgroup of $J(n-r)$.
\end{proof}

Let $W=\langle w_1,w_2,\ldots,w_{n-r}\rangle$ be a fixed complement of $U$ in $V$. Then for each element $\alpha\in\fix(W)$, we can write
\[
	\alpha= \left(
	\begin{array}{cccccccc}
	w_1 & w_2 & \cdots & w_{n-r} & u_1 & u_2 & \cdots & u_r\\
	w_1 & w_2 & \cdots & w_{n-r} & u_1\alpha & u_2\alpha & \cdots & u_r\alpha\\ 
	\end{array}
	\right).
\]
Define a function $\phi: \fix(W)\to GL(U)$ by $\alpha\mapsto\alpha|_U$. It is straightforward to verify that $\phi$ is an isomorphism. Therefore, we obtain the following result.

\begin{proposition}
	Let $W$ be a complement of $U$ in $V$. Then $\fix(W)\cong GL(U)$.
\end{proposition}

By the above results, we have a generating set of $J(n-r)$ depends on generating sets of $\fix(U)$ and $GL(U)$. Refer to W. C. Waterhouse \cite{Waterhouse}, the general linear group $GL(U)$ can be generated by two elements. Now, we will identify a structural property of the subgroup $\fix(U)$ that can be utilized to determine its generating set.

\begin{theorem}
	Let $W$ be a complement of $U$ in $V$. Then the set
	\[
		G(W)=\{\alpha\in\fix(U):W\alpha=W\}
	\]
	is a subgroup of $\fix(U)$ which is isomorphic to $GL(W)$.
\end{theorem}
\begin{proof}
	It is clear that the identity map $1_V$ is the identity element in $G(W)$. Let $\alpha,\beta\in G(W)$. Then $W\alpha\beta=W\beta=W$. Hence $\alpha\beta\in G(W)$. Hence $G(W)$ is a subgroup of $\fix(U)$.

	Define a function $\phi:G(W)\to GL(W)$ by $\alpha\mapsto\alpha|_W$. It is straightforward to verify that $\phi$ is an isomorphism.
\end{proof}

In general, the subgroup $G(W)$ as defined above is not normal in $\fix(U)$. See the following example.

\begin{example}
	Let $V=W\oplus U$ be such that $W=\langle w_1,w_2\rangle$ and $U=\langle u\rangle$. Let $\alpha\in\fix(U)$ and $\beta\in G(W)$ be defined by
	\[
		\alpha= \left(
		\begin{array}{ccc}
		w_1 & w_2 & u\\
		w_1+u & w_2 & u\\ 
		\end{array}
		\right)
	\]
	and
	\[
		\beta= \left(
		\begin{array}{ccc}
		w_1 & w_2 & u\\
		w_2 & w_1 & u\\ 
		\end{array}
		\right).
	\]
	Then $w_1\alpha\beta\alpha^{-1}=(w_1+u)\beta\alpha^{-1}=(w_1\beta+u\beta)\alpha^{-1}=(w_2+u)\alpha^{-1}=w_2\alpha^{-1}+u\alpha^{-1}=w_2+u$ and $w_2\alpha\beta\alpha^{-1}=w_2\beta\alpha^{-1}=w_1\alpha^{-1}=(w_1+u-u)\alpha^{-1}=(w_1+u)\alpha^{-1}-u\alpha^{-1}=w_1-u$ from which it follows that $W\alpha\beta\alpha^{-1}=\langle w_1-u,w_2+u\rangle\neq W$. Hence $\alpha\beta\alpha^{-1}\notin G(W)$. Therefore, $G(W)$ is not normal in $\fix(U)$.
\end{example}

According to \cite{Tingley}, the author gave a constructive proof for counting the number of complements of an $r$-dimension subspace $U$ in an $n$-dimensional vector space $V$ over a finite field $\mathbb{F}$. Let $W$ be a complement of $U$ in $V$ with $\dim W=n-r$. We can write $W=\langle w_1,w_2,\ldots,w_{n-r}\rangle$. In \cite{Tingley}, the author showed that the $n-r$ translator $w_i+U=\{w_1+u'_1,w_2+u'_2,\ldots,w_{n-r}+u'_{n-r}\}$, where $u'_1,u'_2,\ldots,u'_{n-r}\in U$, forms a basis of a complement of $U$ in $V$. Moreover, there is a one-to-one correspondence between $n-r$ tuples $(x_i,x_2,\ldots,x_{n-r})$, $x_i\in w_i+U$ and complements of $U$ in $V$. By using this fact, we obtain the following result.

\begin{theorem}\label{thm: N generated by N(W)UG(W)}
	Let $W=\langle w_1,w_2,\ldots,w_{n-r}\rangle$ be a fixed complement of $U$ in $V$. Define a subset $N(W)$ of $\fix(U)$ by
	\[
		\left\{\left(
			\begin{array}{cccccccc}
			w_1 & w_2 & \cdots & w_{n-r} & u_1 & u_2 & \cdots & u_r\\
			w_1+u'_1 & w_2+u'_2 & \cdots & w_{n-r}+u'_{n-r} & u_1 & u_2 & \cdots & u_r\\ 
			\end{array}
			\right):u'_1,u'_2,\ldots,u'_{n-r}\in U\right\}.
	\]
	Then $\fix(U)=G(W)N(W)$.
\end{theorem}
\begin{proof}
	 Let $\alpha\in\fix(U)$. Then $\alpha\in J(n-r)$ and $\alpha|_U=1_U$. We can write
	\[
		\alpha=\left(
			\begin{array}{cccccccc}
			w_1 & w_2 & \cdots & w_{n-r} & u_1 & u_2 & \cdots & u_r\\
			w_1\alpha & w_2\alpha & \cdots & w_{n-r}\alpha & u_1 & u_2 & \cdots & u_r\\ 
			\end{array}
			\right).
	\]
	We have $W'=\langle w_1\alpha,w_2\alpha,\ldots,w_{n-r}\alpha\rangle$ is a complement of $U$ in $V$. As we mentioned above, it follows that $W'=\langle w_1+u'_1,w_2+u'_2,\ldots,w_{n-r}+u'_{n-r}\rangle$ for some $u'_1,u'_2,\ldots,u'_{n-r}\in U$. We obtain there exist $w'_1,w'_2,\ldots,w'_{n-r}\in V$ such that $w_1+u'_1=w'_1\alpha$, $w_2+u'_2=w'_2\alpha$, $\ldots$, $w_{n-r}+u'_{n-r}=w'_{n-r}\alpha$ since $\alpha$ is surjective. Hence
	\[
		\alpha=\left(
			\begin{array}{cccccccc}
			w'_1 & w'_2 & \cdots & w'_{n-r} & u_1 & u_2 & \cdots & u_r\\
			w_1+u'_1 & w_2+u'_2 & \cdots & w_{n-r}+u'_{n-r} & u_1 & u_2 & \cdots & u_r\\ 
			\end{array}
			\right).
	\]
	Moreover, $W=\langle w'_1,w'_2,\ldots,w'_{n-r}\rangle$ since $\alpha$ is injective. Define $\beta\in G(W)$ and $\gamma\in N(W)$ by
	\[
		\beta=\left(
			\begin{array}{cccccccc}
			w'_1 & w'_2 & \cdots & w'_{n-r} & u_1 & u_2 & \cdots & u_r\\
			w_1 & w_2 & \cdots & w_{n-r} & u_1 & u_2 & \cdots & u_r\\ 
			\end{array}
			\right)
	\]
	and
	\[
		\gamma=\left(
			\begin{array}{cccccccc}
			w_1 & w_2 & \cdots & w_{n-r} & u_1 & u_2 & \cdots & u_r\\
			w_1+u'_1 & w_2+u'_2 & \cdots & w_{n-r}+u'_{n-r} & u_1 & u_2 & \cdots & u_r\\ 
			\end{array}
			\right).
	\]
	It is clear that $\alpha=\beta\gamma$.
\end{proof}

Finally, we give some properties of $N(W)$.

\begin{theorem}
	The subset $N(W)$ as in Theorem \ref{thm: N generated by N(W)UG(W)} is a normal subgroup of $\fix(U)$ such that $G(W)\cap N(W)=\{1_V\}$. Additionally, $N(W)$ is isomorphic to the direct sum of $n-r$ copies of $U$ as an additive group.
\end{theorem}
\begin{proof}
	Let $W=\langle w_1,w_2,\ldots,w_{n-r}\rangle$. Then $V=\langle w_1,w_2,\ldots,w_{n-r}\rangle\oplus U$.
	
	We first show that $N(W)$ is a subgroup. Let $\alpha,\beta\in N(W)$. We can write
	\[
		\alpha=\left(
			\begin{array}{cccccccc}
			w_1 & w_2 & \cdots & w_{n-r} & u_1 & u_2 & \cdots & u_r\\
			w_1+u'_1 & w_2+u'_2 & \cdots & w_{n-r}+u'_{n-r} & u_1 & u_2 & \cdots & u_r\\ 
			\end{array}
			\right)
	\]
	and
	\[
		\beta=\left(
			\begin{array}{cccccccc}
			w_1 & w_2 & \cdots & w_{n-r} & u_1 & u_2 & \cdots & u_r\\
			w_1+u''_1 & w_2+u''_2 & \cdots & w_{n-r}+u''_{n-r} & u_1 & u_2 & \cdots & u_r\\ 
			\end{array}
			\right)
	\]
	for some $u'_i,u''_i\in U$. Then, for each $w_i\in\{w_1,w_2,\ldots,w_{n-r}\}$, we obtain
	\[
		w_i\alpha\beta=(w_i+u'_i)\beta=w_i\beta+u'_i\beta=w_i+u''_i+u'_i
	\]
	which implies that
	\[
		\alpha\beta=\left(
			\begin{array}{cccccccc}
			w_1 & w_2 & \cdots & w_{n-r} & u_1 & u_2 & \cdots & u_r\\
			w_1+u'''_1 & w_2+u'''_2 & \cdots & w_{n-r}+u'''_{n-r} & u_1 & u_2 & \cdots & u_r\\ 
			\end{array}
			\right)
	\]
	where $u'''_i=u'_i+u''_i\in U$ for all index $i$. It follows that $\alpha\beta\in N(W)$.

	To prove normality, let $\alpha\in\fix(U)$ and $\beta\in N(W)$. We can write
	\[
		\alpha=\left(
			\begin{array}{cccccccc}
			w_1\alpha^{-1} & w_2\alpha^{-1} & \cdots & w_{n-r}\alpha^{-1} & u_1 & u_2 & \cdots & u_r\\
			w_1 & w_2 & \cdots & w_{n-r} & u_1 & u_2 & \cdots & u_r\\ 
			\end{array}
			\right)
	\]
	and
	\[
		\beta=\left(
			\begin{array}{cccccccc}
			w_1 & w_2 & \cdots & w_{n-r} & u_1 & u_2 & \cdots & u_r\\
			w_1+u'_1 & w_2+u'_2 & \cdots & w_{n-r}+u'_{n-r} & u_1 & u_2 & \cdots & u_r\\ 
			\end{array}
			\right)
	\]
	for some $u'_1,u'_2,\ldots,u'_{n-r}\in U$. For each $w\in\{w_1,w_2,\ldots,w_{n-r}\}$, we have $w=\sum a_iw_i\alpha^{-1}+\sum b_ju_j$ for some scalars $a_i,b_j$. It follows that 
	\begin{eqnarray*}
		w\alpha\beta\alpha^{-1}&=&\sum a_iw_i\alpha^{-1}\alpha\beta\alpha^{-1}+\sum b_ju_j\alpha\beta\alpha^{-1}\\
		&=&\sum a_iw_i\beta\alpha^{-1}+\sum b_ju_j\\
		&=&\sum a_i(w_i+u'_i)\alpha^{-1}+\sum b_ju_j\\
		&=&\sum a_iw_i\alpha^{-1}+\sum a_iu'_i\alpha^{-1}+\sum b_ju_j\\
		&=&w+\sum a_iu'_i\\
		&=&w+u''
	\end{eqnarray*}
	for some $u''\in U$. We can write
	\[
		\alpha\beta\alpha^{-1}=\left(
			\begin{array}{cccccccc}
			w_1 & w_2 & \cdots & w_{n-r} & u_1 & u_2 & \cdots & u_r\\
			w_1+u''_1 & w_2+u''_2 & \cdots & w_{n-r}+u''_{n-r} & u_1 & u_2 & \cdots & u_r\\ 
			\end{array}
			\right)
	\]
	for some $u''_1,u''_2,\ldots,u''_{n-r}\in U$. Thus $\alpha\beta\alpha^{-1}\in N(W)$.

	To show that $G(W)\cap N(W)=\{1_V\}$, let $\alpha\in G(W)\cap N(W)$. Then we can write
	\[
		\alpha=\left(
			\begin{array}{cccccccc}
			w_1 & w_2 & \cdots & w_{n-r} & u_1 & u_2 & \cdots & u_r\\
			w_1+u'_1 & w_2+u'_2 & \cdots & w_{n-r}+u'_{n-r} & u_1 & u_2 & \cdots & u_r\\ 
			\end{array}
			\right)
	\]
	for some $u'_1,u'_2,\ldots,u'_{n-r}\in U$. It is clear that $\langle w_1,w_2,\ldots,w_{n-r}\rangle=\langle w_1+u'_1,w_2+u'_2,\ldots,w_{n-r}+u'_{n-r}\rangle$ if and only if $u'_1=u'_2=\cdots=u'_{n-r}=0$. Hence $\alpha=1_V$. 

	Finally, we prove that $N(W)$ is isomorphic to the direct sum of $n-r$ copies of $U$ as an additive group. Define a function $\phi:N(W)\to U^{n-r}$ by
	\[
		\left(
			\begin{array}{cccccccc}
			w_1 & w_2 & \cdots & w_{n-r} & u_1 & u_2 & \cdots & u_r\\
			w_1+u'_1 & w_2+u'_2 & \cdots & w_{n-r}+u'_{n-r} & u_1 & u_2 & \cdots & u_r\\ 
			\end{array}
		\right)\mapsto(u'_1,u'_2,\ldots,u'_{n-r}).
	\]
	It is routine matter to verify that $\phi$ is an isomorphism.
\end{proof}

The following corollary is immediate from the above two theorems.

\begin{corollary}
	$\fix(U)$ is an internal semidirect product of $N(W)$ and $G(W)$, that is, $\fix(U)=G(W)\ltimes N(W)$.
\end{corollary}

It is well-known that any $r$-dimensional vector space $U$ over a finite field $\mathbb{F}$ is isomorphic to $\mathbb{F}^r$. Thus, we have the following corollary.

\begin{corollary}
	$N(W)$ is isomorphic to the direct sum of $n-r$ copies of $\mathbb{F}^r$ as an additive group.
\end{corollary}


\subsection*{Acknowledgments} This research was supported by Chiang Mai University.

\subsection*{Conflict of interest} The author declares no conflict of interest.

\bibliographystyle{abbrv}\addcontentsline{toc}{section}{References}
\bibliography{References}

\begin{flushleft}
\vskip.3in

KRITSADA SANGKHANAN, Department of Mathematics, Faculty of Science, Chiang Mai University, Chiang Mai, 50200, Thailand; e-mail: kritsada.s@cmu.ac.th

\end{flushleft}
\end{document}